\documentclass[11pt,a4paper,reqno]{amsart}
\usepackage{mathptmx}
\usepackage{amssymb}
\usepackage{amsmath}
\usepackage{amsthm}
\usepackage{latexsym}
\usepackage{amsfonts}
\usepackage{mathrsfs}
\usepackage{setspace}
\usepackage{xcolor}
\usepackage{MnSymbol}
\usepackage[colorlinks=true,linkcolor=blue,urlcolor=blue,bookmarksnumbered,citecolor=blue]{hyperref}

\allowdisplaybreaks

 \newtheorem{thm}{Theorem}[section]

 \newtheorem{prop}[thm]{Proposition}
 \theoremstyle{definition}
 \newtheorem{defn}[thm]{Definition}
 \theoremstyle{remark}
 \newtheorem{rem}[thm]{Remark}
 
 \numberwithin{equation}{section}

\numberwithin{equation}{section}
\let\oldboxplus\boxplus
\renewcommand{\boxplus}{%
  \mathrel{\raisebox{.5pt}{$\oldboxplus$}}%
}

\renewcommand{\emptyset}{\varnothing}

\newcommand{\R}{\ensuremath{\mathbb R}}    % Reelle Zahlen
\newcommand{\C}{\ensuremath{\mathbb C}}    % Komplexe Zahlen
\newcommand{\N}{\ensuremath{\mathbb N}}    % Nat"urliche Zahlen
\newcommand{\<}{\langle}
\renewcommand{\>}{\rangle}

\newcommand{\product}{[\cdot\,,\cdot]}
\newcommand{\hproduct}{(\cdot\,,\cdot)}

\newcommand{\aproduct}{\langle\cdot\,,\cdot\rangle}
\newcommand{\calH}{\mathcal H}

\newcommand{\calU}{\mathcal U}

\newcommand{\la}{\lambda}
\newcommand{\veps}{\varepsilon}

\newcommand{\mat}[4]
{
   \begin{pmatrix}
      #1 & #2\\
      #3 & #4
   \end{pmatrix}
}

\newcommand{\smallmat}[4]{\left(\begin{smallmatrix}#1 & #2\\#3 & #4\end{smallmatrix}\right)}

\renewcommand{\Im}{\operatorname{Im}}

\renewcommand{\ker}{\operatorname{ker}}
\newcommand{\ran}{\operatorname{ran}}
\newcommand{\dom}{\operatorname{dom}\,}

\newcommand{\sap}{\sigma_{{ap}}}

\renewcommand{\sp}{\sigma_{+}}
\newcommand{\sm}{\sigma_{-}}

\newcommand{\spp}{\sigma_+}
\newcommand{\smm}{\sigma_-}

\newcommand{\downto}{\downarrow}

\newcommand{\ol}{\overline}
\newcommand{\ds}{\dotplus}

\newcommand{\wh}{\widehat}

\newcommand{\sgn}{\operatorname{sgn}}
\newcommand{\dist}{\operatorname{dist}}

\title[On non-negative operators in Krein spaces and their perturbations]{On non-negative operators in Krein spaces and their perturbations}

\author{Jussi Behrndt}
\address{Jussi Behrndt: Institut f\"ur Angewandte Mathematik, Technische Universit\"at Graz,
Austria, Email: \textsc{behrndt@tugraz.at}.}%
\author{Friedrich M.~Philipp}
\address{Friedrich M. Philipp: Institute of Mathematics,
Technische Universit\"at Ilmenau, Germany, Email: \textsc{friedrich.philipp@tu-ilmenau.de}.}
\author{Carsten Trunk}
\address{Carsten Trunk: Institute of Mathematics,
Technische Universit\"at Ilmenau, Germany, Email:\textcolor{white}{hu }\textsc{carsten.trunk@tu-ilmenau.de}.}

\dedicatory{In memory of Heinz Langer – an outstanding mathematician and unique personality}

\begin{document}
\begin{abstract}
One of the most important contributions of Heinz Langer in the area of operator theory in Krein spaces is the introduction of the notion of definitizable operators
and the construction of the corresponding spectral function.
In this note we obtain a new characterization for the subclass of 
non-negative operators in Krein spaces which is based on local sign type properties of the spectrum and growth conditions on the resolvent. 
Based on these local properties, a notion of local non-negativity for self-adjoint operators in Krein spaces is defined and it is shown that such classes of operators appear naturally as perturbations of non-negative operators. 
\end{abstract}

%%% ----------------------------------------------------------------------
\maketitle
%%% ----------------------------------------------------------------------


\smallskip
\noindent \textbf{Mathematics subject classications.} Primary 47B50; Secondary 46C20
%37M99, %
    %47B32, %
    %65C05, %
    %65D12 %

\bigskip

\section{Introduction}
\noindent Spectral theory of self-adjoint operators in Krein spaces is a challenging and advanced field in modern 
mathematical analysis with various applications to differential equations and mathematical physics, see, e.g., 
\cite{BC23,G12,GGH13,GGH15,J91,J93,J00,LNT06,LNT08,N80-1,N80-2,N83,V91} for the Klein-Gordon equation,
\cite{ALM01,AMT10,DL96,JT07,JTW08,LcMM,L73,L74,LLMT08,lmm,LMM06,LT04,LS17,M88} for operator pencils and related problems, 
\cite{B07-a,B13,bp,BPT28,BSTT24,bst19,cl,CN94,cn,DL77,DL86,F90,KK08,KKM09,KT09,kmwz,P95,P00,z} for ordinary and partial 
differential operators with indefinite weights,
and in the context of completeness of the system of eigenfunctions and associated functions of indefinite Sturm-Liouville problems, \cite{AP,Bea85,cu,CDT21,CFK13,cl,CN94,cn,Fl95,Fl98,Fl08,Fl15,Kost11,Kost13,Par03,Py3}.

Heinz Langer is one of the pioneers in this area and has shaped the field with his groundbreaking contributions going back to the habilitation thesis \cite{L65} in 1965, where the fundamental concept of definitizable self-adjoint operators in Krein spaces was introduced, and the spectral calculus for this important class of operators was developed.

Let us familiarize ourselves with the subject by starting with more special types of operators and then generalizing to the class of definitizable operators and beyond. In fact, first of all it is important to realize that an arbitrary self-adjoint operator $A$ with domain $\dom A$ in a Krein space $(\mathcal H,[\cdot,\cdot])$ is a very general object with little intrinsic structure: although its spectrum $\sigma(A)$ is symmetric with respect to the real axis, it need not be real and may even coincide with the entire complex plane.

Therefore, in order to establish a fruitful theory, additional conditions are needed. A particularly simple, but still natural and interesting case, appears if one assumes that the self-adjoint operator $A$ is 
{\it non-negative} with respect to the indefinite metric $[\cdot,\cdot]$, i.e.\ 
$$
[Af,f]\geq 0 \quad \mbox{for all }  f\in\dom A
$$ 
and the resolvent set $\rho(A)$ is non-empty. Such an operator admits a spectral function with possible singularities at the critical points $0$ and $\infty$,  the 
spectrum is contained in $\R$ and possesses some additional sign properties, namely, the spectral points in $(0,\infty)$ are of positive type and in $(-\infty,0)$ of negative type;
this goes back to \cite{KS66}, see also \cite{A79,b83,L82}. 
A straightforward generalization of the class of non-negative operators are operators with {\it finitely many negative squares}, that is, the form $[A\cdot,\cdot]$ is non-negative only on a 
subspace of $\dom A$ with finite codimension. 
The next step brings us to the class of definitizable operators and the fundamental 
contributions to operator theory in Krein spaces by Heinz Langer in the 1960s. Recall that a self-adjoint operator $A$ in a
Krein space $(\mathcal H,[\cdot,\cdot])$ is \textit{definitizable} if there exists a real polynomial $p\not= 0$ 
such that 
$$
[p(A)f,f]\geq 0 \quad  \mbox{for all } f\in\dom p(A)
$$
and the resolvent set $\rho(A)$ is non-empty.
Such an operator admits a spectral function, and with the help of this spectral function the real points of the spectrum $\sigma(A)$ can be classified in points of positive and negative type, and a finite set of critical points. Furthermore, the non-real spectrum consists of at most finitely many pairs of eigenvalues (which are symmetric with respect to the real line) and the growth of the resolvent of $A$ towards  real points is of finite order.
The classical paper \cite{L82} is an excellent source for an introduction into the theory of definitizable operators and their applications, which also provides the derivation and the construction of the spectral function; we also refer the reader to the very recent monograph \cite{G22} for more details and further references. 
Another substantial major step forward was taken in the paper \cite{lmm}, where spectral points of positive and negative type for self-adjoint operators in Krein spaces were characterized via approximative eigensequences (see Definition~\ref{d:spp}), which paved the way to local spectral analysis and will play an essential role here.

The main objective of this note is to view non-negative self-adjoint operators in Krein spaces in a local spirit. In this context we provide in Section 3 a complete characterization of non-negative operators via local spectral properties and resolvent growth conditions near $0$ and $\infty$, as well as an additional non-negativity condition related to the spectral point $0$; cf. Theorem~\ref{t:nn}. To the best of our knowledge such an explicit characterization focusing on local spectral properties is not contained in the mathematical literature. In the case that $0$ is not a singular critical point, our conditions simplify and actually reduce to a non-negativity condition for the root vectors in Theorem~\ref{t:nn2}. 
If, in addition, $\infty$ is also not a singular critical point and the geometric and algebraic eigenspaces at $0$ coincide, that is, $\ker A=\ker A^2$, then $A$ is similar to a self-adjoint operator in a Hilbert space; cf. Theorem~\ref{t:nn3}. This situation is often treated in the mathematical literature in the context of differential operators.

Our point of view is particularly convenient for perturbation problems. More precisely, it is clear that non-negativity is not stable with respect to 
(relatively) bounded additive perturbations, but intuitively one may still expect some similar spectral behaviour outside sufficiently large compact sets in $\mathbb C$ if the perturbation is bounded or at least relatively bounded in a suitable sense. In fact, this naturally leads to classes of self-adjoint operators in Krein spaces that are non-negative outside a compact set, and these operators can
be characterized in a convenient way by our local analysis; cf.\ Section 4 for more details.

In Section~\ref{555} we then interpret and illustrate the concept of local non-ne\-ga\-ti\-vi\-ty in the context of perturbation theory for self-adjoint operators in Krein spaces. E.g., the simplest and well studied case is when the unperturbed operator $A$ is fundamentally reducible (that is, $A$ commutes with some fundamental symmetry)
and the additive perturbation $V$ is bounded and self-adjoint.
Then  Theorem~\ref{l:lmm} in Section~\ref{555}
states that $A+V$ is locally non-negative outside 
a capsule-shaped region around zero, where the shape
is determined by $V$, for details see Section~\ref{555}.
This type of results have a long history and go back (at least) to the paper \cite{L67} by Heinz Langer, see also \cite{LLMT05,lmm,T08,T09} and \cite{BPT28,GLMPT20,p23} for more recent generalizations. However, the fundamental decomposition that reduces $A$ is typically unknown in applications, and more advanced estimates and techniques are required to obtain the perturbation results we present here as Theorem \ref{t:main1} and Theorem \ref{t:caot_main} which are taken from \cite{BPT28} and \cite{p23}, respectively.
Such abstract perturbation results can be applied, e.g. 
to singular Sturm-Liouville operators
with indefinite weight functions and $L^p$-potentials; cf. \cite{BPT28,p23}
for more details.

\bigskip\noindent
\textbf{Notation.}
$\mathbb C^+$ ($\mathbb C^-$) denotes the open upper half-plane (the open lower half-plane, respectively) and  $\mathbb R^+$ ( $\mathbb R^-$) the set of positive (negative,  respectively) real numbers, i.e.\ $\mathbb R^+:= (0,\infty)$ and $\mathbb R^-:=(-\infty , 0)$.  The compactification $\mathbb R \cup \{\infty\}$ of $\mathbb R$ is denoted by $\overline{\mathbb R}$ and
the compactification $\mathbb C \cup \{\infty\}$ of $\mathbb C$ by $\overline{\mathbb C}$. 

Let $T$ be a linear operator in a Hilbert space.
The domain, kernel, and range of $T$ will be denoted by $\dom T$, $\ker T$, and $\ran T$, respectively.
For a closed linear operator $T$, we denote the spectrum by $\sigma(T)$ and the resolvent set by 
$\rho(T)$. 
A point $\lambda\in\C$ belongs to the {\it approximate point spectrum} $\sap(T)$ of $T$ if there exists a sequence $(f_n)_{n\in\N}$ in $\dom T$ with $\|f_n\| = 1$ for $n\in\N$ and $(T - \la)f_n\to 0$ as $n\to\infty$. Note that both the point spectrum and the continuous spectrum of $T$ are contained in $\sap(T)$.

\bigskip\noindent
\textbf{Acknowledgments.}
This research was funded in part by the Austrian Science Fund (FWF) 10.55776/P 33568-N. FMP gratefully acknowledges support from the German Research Foundation (DFG), project numbers 554600805 and 519323897.% The authors thank the referee for a very careful reading of the manuscript.
%For the purpose of open access, the authors have applied a CC BY public copyright license to any Author Accepted Manuscript version arising from this submission.

\section{Preliminaries on self-adjoint operators in Krein spaces}

Let $(\calH,\product)$ be a Krein space, let $J$ be a fixed fundamental symmetry in $\calH$, and denote by $\hproduct$ the Hilbert space scalar product induced by $J$, i.e.\ $\hproduct = [J\,\cdot\,,\cdot]$. The induced norm is denoted by $\|\cdot\|$.
 For a detailed treatment of Krein spaces and operators therein we refer to the monographs \cite{ai,b,G22}.

For a densely defined linear operator $A$ in $\calH$ the adjoint with respect to the Krein space inner product $[\cdot,\cdot]$ is denoted by $A^+$. We mention that $A^+ = JA^*J$, where $A^*$ denotes the adjoint of $A$ with respect to the scalar product $\hproduct$. The operator $A$ is called {\it symmetric {\rm (}self-adjoint{\rm )} in the Krein space $(\calH,\product)$} if $A\subset A^+$ ($A = A^+$, respectively). Occasionally, we write $\product$-symmetric or $\product$-self-adjoint. The latter is equivalent to self-adjointness of the operator $JA$ in the Hilbert space $(\calH,\hproduct)$.

As mentioned in the Introduction, the spectrum of a self-adjoint operator $A$ in a Krein space is symmetric with respect to the real axis, but is in general not contained in $\R$. Furthermore, it is known that its real spectral points belong to the approximate point spectrum $\sap(A)$ (see, e.g., \cite[Corollary~VI.6.2]{b}):
\begin{equation}\label{e:Rsssap}
\sigma(A)\cap\R\subset\sap(A).
\end{equation}
The following definition can be found in, e.g., \cite{j03,LcMM,lmm}.

\begin{defn}\label{d:spp}
Let $A$ be a self-adjoint operator in the Krein space $(\calH,\product)$. A point $\la\in\sap(A)$ is called a {\em spectral point of positive {\rm (}negative{\rm )} type} of $A$ if for every sequence $(f_n)$ in $\dom A$ with $\|f_n\| = 1$ and $(A - \la)f_n\to 0$ as $n\to\infty$ we have
$$
\liminf_{n\to\infty}\,[f_n,f_n] > 0\quad\Big(\limsup_{n\to\infty}\,[f_n,f_n] < 0,\;\text{respectively}\Big).
$$
The set of all spectral points of positive {\rm (}negative{\rm )} type of $A$ will be denoted by $\spp(A)$ {\rm (}$\smm(A)$, respectively{\rm )}. A set $\Delta\subset\C$ is said to be of {\em positive {\rm (}negative{\rm )} type} with respect to $A$ if each spectral point of $A$ in $\Delta$ is of positive type {\rm (}negative type, respectively{\rm )}, and it is called of {\em definite type} with respect to $A$ if it is either of positive or of negative type.
\end{defn}

The sets $\spp(A)$ and $\smm(A)$ are contained in $\R$, open in $\sigma(A)$ and the non-real spectrum of $A$ cannot accumulate to $\spp(A)\cup\smm(A)$,
see \cite{ajt,j03,lmm}. At a spectral point $\la_0$ of positive or negative type of a self-adjoint operator $A$ in a Krein space the growth of the resolvent of $A$ is of order one in the sense of the following definition; cf.\ \cite{ajt,j03,lmm}.

\begin{defn}\label{d:growth}
Let $A$ be a self-adjoint operator in the Krein space $(\calH,\product)$.
\begin{itemize}
\item[\rm (i)] We say that the growth of the resolvent of $A$ at $\la_0\in \mathbb R$ is of order $m\ge 1$ if there exist an open neighborhood $\calU \subset \mathbb C$ of $\la_0$  and $M > 0$ such that $\calU\setminus\R\subset\rho(A)$ and
\begin{equation}\label{e:growth}
\|(A - \lambda)^{-1}\| \le \frac{M}{|\Im\la|^{m}},
\quad \la\in\calU\setminus\mathbb R.
\end{equation}
\item[\rm (ii)] We say that the growth of the resolvent of $A$ at $\infty$ is of order $m\ge 1$ if there exist an open neighborhood $\calU$ of $\infty$ in $\ol\C$ and $M > 0$ such that $\calU\setminus\ol\R\subset\rho(A)$ and
\begin{equation}\label{e:growth2}
\|(A - \lambda)^{-1}\| \le 
\frac{M |\la|^{2m-2}}{|\Im\la|^{m}}, \quad \la\in\calU\setminus\ol\R.
\end{equation}
\end{itemize}
\end{defn}

It is clear that if the growth of the resolvent of $A$ at some point $\lambda_0 \in \ol{\mathbb R}$ is of order $m$, then it is also of order
 $n$ for $n>m$.

\begin{rem}
Typically, the two separate resolvent growth conditions \eqref{e:growth} and \eqref{e:growth2} appear in 
the literature in the summarized form
\begin{equation*}
\|(A - \lambda)^{-1}\| \le 
M\frac{(1+ |\la|)^{2m-2}}{|\Im\la|^{m}}, \quad \la\in\calU\setminus\ol\R,
\end{equation*}
see, e.g. \cite{J88,J91,j03}. However, for our purposes it is slightly more convenient to treat
finite points and $\infty$ separately as in
\eqref{e:growth} and \eqref{e:growth2}.
\end{rem}

We briefly recall the notion of locally
definitizable self-adjoint operators.
Such classes of operators appeared first in
a paper by Langer in 1967 (see \cite{L67})  without having a name
at that time. Later, in a series of papers, Jonas studied these
operators and introduced the notion of locally definitizable
operators; cf.\ \cite{J86,J88,J91,j03}.
In the following definition let
$\Omega$ be some domain in $\overline{\mathbb C}$ symmetric with respect
to the real axis such that $\Omega\cap\overline{\mathbb R}\not=\emptyset$
and the intersections of $\Omega$ with the upper and lower open
half-planes are simply connected.

\begin{defn}\label{locdef}
Let $A$ be a self-adjoint operator in the Krein space $(\calH,\product)$ such
that $\sigma(A)\cap(\Omega \backslash\overline{\mathbb R})$ consists of
isolated points which are poles of the resolvent of $A$, and no
point of $\Omega\cap\overline{\mathbb R}$ is an accumulation point of the
non-real spectrum of $A$ in $\Omega$. The operator $A$ is said
to be {\it definitizable over} $\Omega$, if the following holds.
\begin{itemize}
\item[{\rm (i)}]
Every point $\mu\in\Omega\cap\overline{\mathbb R}$ has an open connected
neighborhood $I_\mu$ in $\overline{\mathbb R}$ such that both components of
$I_\mu\backslash\{\mu\}$ are of definite type with respect to $A$, respectively.
\item[{\rm (ii)}]
For every $\lambda_0 \in \Omega \cap \overline{\mathbb R}$ the growth of
the resolvent of $A$ is of finite order.
\end{itemize} 
The points in $\sigma(A)\cap\Omega\cap\R$ that do not belong to $\sp(A)\cup\sm(A)$ are called  
{\em critical points} of $A$ in $\Omega$. If $\infty\in\Omega$, then $\infty$ is called
{\em critical point} of $A$ if both $\sigma_+(A)$ and $\sigma_-(A)$ 
accumulate at $\infty$, and
one component of $I_\infty\backslash\{\infty\}$
is of positive type, and the other component of $I_\infty\backslash\{\infty\}$ is of negative type 
with respect to $A$.
\end{defn}

Let $A$ be definitizable over $\Omega$. Then $A$ possesses a
\emph{local spectral function} $\Delta\mapsto E(\Delta)$ on
$\Omega\cap\overline{\mathbb R}$, which is defined for all Borel subsets $\Delta$ of $\Omega\cap\overline{\mathbb R}$ whose $\ol\R$-boundary points are of definite type with respect to $A$ and belong to $\Omega\cap\overline{\mathbb R}$, \cite[Section 3.4 and Remark 4.9]{j03}. For such a set $\Delta$ we collect some properties of $E(\Delta)$ in the following theorem, see \cite[Section 3.4 and Remark 4.9]{j03}.

\begin{thm}\label{locspecfunc}
Let $A$ be a self-adjoint operator in the Krein space $(\calH,\product)$,
assume that $A$ is definitizable over $\Omega$, and let 
$\Delta\mapsto E(\Delta)$ be the local spectral function on
$\Omega\cap\overline{\mathbb R}$.
Then the spectral projection $E(\Delta)$ is a bounded $\product$-self-adjoint projection
with the following properties:
\begin{enumerate}
\item[{\rm (a)}] $E(\Delta)$  commutes
with every bounded operator which commutes with the resolvent of $A$.
\item[{\rm (b)}] $\sigma(A|E(\Delta)\calH)\,\subset\,\sigma(A)\cap\ol\Delta$.
\item[{\rm (c)}] 
If $\overline\Delta$ is of positive type, then $(E(\Delta)\calH,\product)$ is a Hilbert space. Similarly, if $\overline{\Delta}$ is of negative
type, then $(E(\Delta)\calH,-\product)$ is a Hilbert space. 
\item[{\rm (d)}] 
$\sigma(A|(I - E(\Delta))\calH)\,\subset\,\sigma(A)\setminus\operatorname{int}(\Delta)$, where
$\operatorname{int}(\Delta)$ denotes the interior of $\Delta$ with respect to the topology of
$\ol\R$.
\item[{\rm (e)}]  If, in addition, $\Delta$ is a neighborhood of $\infty$
{\rm (}with respect to the topology of $\ol\R${\rm )}, then $A|(I-E(\Delta)) \calH$ is a bounded operator.
\end{enumerate}
\end{thm}

The local spectral function of a definitizable operator 
$A$ over $\Omega$ allows forn an equivalent definition of spectral points of definite type which is used in \cite{L82}:
A real point $\lambda \in \sigma(A)\cap\Omega$
belongs to $\sigma_+(A)$ ($\sigma_-(A)$) if and only if there exists an open interval $\Delta\subset \mathbb R$,
$\lambda \in \Delta$,
such that $E(\Delta)$ is defined and $(E(\Delta)\calH, \product)$ (resp., $(E(\Delta)\calH, -\product))$ is a Hilbert space, see, e.g.,  \cite[Proposition 25]{ajt}.

In the sequel we shall often view restrictions 
$$A'=A|E(\Delta)\calH\qquad \text{and}\qquad A''=A|(I-E(\Delta))\calH $$ 
of a locally definitizable operator $A$ to spectral subspaces as operators acting in these subspaces, in which case we write 
$$
A=\begin{pmatrix} A' & 0 \\ 0 & A''\end{pmatrix},\qquad \calH=E(\Delta)\calH\,[\ds]\,(I-E(\Delta))\calH;
$$
however, sometimes it is also convenient to interpret restrictions of $A$ as operators in the same space $\calH$, so that, e.g. the sum $A=A|E(\Delta)\calH + A|(I-E(\Delta)) \calH$ is well-defined.

Next, we classify the critical points of $A$ in $\Omega$
%into
%points in $\sigma(A)\cap\Omega\cap\ol\R$ which are not in $\sp(A)\cup\sm(A)$ 
%as so-called 
as {\it regular} or {\it singular}.% critical points.

\begin{defn}
Let $A$ be a self-adjoint operator in the Krein space $(\calH,\product)$,
assume that $A$ is definitizable over $\Omega$, and let 
$\Delta\mapsto E(\Delta)$ be the local spectral function on
$\Omega\cap\overline{\mathbb R}$. 
A critical point $\mu\in\Omega\cap\R$ is called {\it regular} if there exists $C > 0$ 
such that $\|E([\mu-\veps,\mu+\veps])\|\le C$ 
for all sufficiently small $\veps > 0$. 
If $\infty\in\Omega$ is a critical point, then $\infty$ is called {\it regular} if there exists $C > 0$ 
such that  $\|E(\overline{\mathbb R}\setminus [-r,r])\|\le C$  for all sufficiently large $r > 0$. 
Critical points, which are not regular, are called {\it singular}.
\end{defn}

By \cite[Theorem 4.7]{j03}, a self-adjoint operator $A$ is
definitizable over $\overline{\mathbb C}$ if and only if $A$ is {\it
definitizable} in the classical sense of Heinz Langer \cite{L65,L82}, that is, the resolvent set 
$\rho(A)$ of $A$ is non-empty and
there exists a real polynomial $p\not= 0$  such that $p(A)$ is a non-negative operator
in $\mathcal H$, i.e.,
\begin{equation}\label{pa}
[p(A)f,f] \geq 0\quad  \mbox{for all } f\in\dom p(A).
\end{equation}
It was already shown in \cite{L65,L82} (see also \cite{b87}) that a definitizable operator possesses a spectral function on $\overline{\mathbb R}$ with a possible finite set of singularities (which are the singular critical points), and the spectral projections are defined for all Borel subsets $\Delta$ of $\overline{\mathbb R}$
whose $\ol\R$-boundary points are of definite type (see also \cite[Chapter~11]{G22}). This spectral function coincides with the (local) spectral function mentioned above, and hence has the properties in Theorem~\ref{locspecfunc}.

\section{Spectral characterization of non-negative operators}
\label{Grunewald}

A self-adjoint operator $A$ in a Krein space $(\calH, \product)$ is said to be {\it non-negative} if  $\rho(A)\neq\emptyset$ and 
\begin{equation}\label{ojabitte3}
[A f,f] \geq 0,\quad  \mbox{for all } f\in\dom A.
\end{equation}
If, for some $\gamma > 0$, $[Af,f]\geq\gamma\|f\|^2$ holds for all $f\in\dom A$, then $A$ is called  {\it uniformly positive}. Note that a self-adjoint operator $A$  is uniformly positive if and only if $A$ is non-negative and $0\in\rho(A)$.

The next theorem characterizes non-negativity of a self-adjoint operator in a Krein space in terms of its local spectral properties. The interesting and new observation here is the sufficiency of the conditions (i)--(iii) below for $A$ to be non-negative; their necessity is known. In fact, 
as $A$ is definitizable with definitizing polynomial $p(\xi)=\xi$, $\xi\in\R$, (see \eqref{pa}) it possesses a spectral function $E$ on $\ol \R$ as in Theorem \ref{locspecfunc} and 
(i)--(iii)  can be observed as special cases of properties of definitizable operators from \cite{L82} and the spectral points of definite type \cite{lmm} together with \cite[Proposition 25]{ajt};
note also that the only possible critical points of $A$ are $0$ and $\infty$ (see \cite{A79,J81,KS66,L65} for more details). However, there exists also 
an elementary direct proof for the necessity part, which we will present here.

\begin{thm}\label{t:nn}
A self-adjoint operator $A$ in the Krein space $(\calH,\product)$ is non-negative if and only if the following conditions are satisfied:
\begin{enumerate}
\item[{\rm (i)}]   $\sigma(A)\subset\R$ and $\sigma(A)\cap\R^\pm\subset\sigma_\pm(A)$.
\item[{\rm (ii)}]  The growth of the resolvent of $A$ at $\infty$ is of order  $2$.
\item[{\rm (iii)}] The growth of the resolvent of $A$ at $0$ is of order $2$, and for each sequence $(f_n)$ in $\dom(A^2)$ with $A^2f_n/\|f_n\|\to 0$ as $n\to\infty$ we have
$$
\liminf_{n\to\infty}\,[Af_n,f_n]\ge 0.
$$
\end{enumerate}
\end{thm}

\begin{proof}
Let us assume that $A$ is a self-adjoint operator that satisfies the 
above conditions (i)--(iii). 
\vskip 0.2cm
\noindent
{\it Step 1.} We observe in this first step that $A$ is definitizable over $\ol\C$ in the sense
of Definition~\ref{locdef}. In fact, from (i) it is clear that for every point $\mu\in\R^+$ ($\mu\in\R^-$) there exists an open neighborhood
in $\R^+$ ($\R^-$) which is of positive type (negative type, respectively), 
and for $0$ and $\infty$ there exist neighborhoods where both components are of definite (but different) type. Furthermore, it is well known that the growth of the resolvent of $A$ near 
spectral points of positive or negative type is of order one (see, e.g. \cite{lmm}), and according to
conditions (ii) and (iii) the growth of the resolvent of $A$ near $0$ and $\infty$ 
is of order $2$. Thus, $A$ is definitizable over $\ol\C$ and hence definitizable 
in the classical sense of Langer, see \cite[Theorem~4.7]{j03}. 
\vskip 0.2cm
\noindent
{\it Step 2.} 
Let $E$ be the spectral function of $A$ and consider the self-adjoint operator 
$$A_0 = A|E((-1,1))\calH.$$ 
It is clear that the conditions (i)--(iii) hold also for $A_0$ 
and, in addition, $A_0$ is bounded. In this step we prove that $A_0$ is non-negative. For this
it is convenient to set $\Delta_n=\big[-\frac{1}{n},\frac{1}{n}\big]$ and to consider the bounded self-adjoint operators 
\begin{equation*}
A_0 | E(\Delta_n)\mathcal H,\quad A_0\vert E((-1,-\tfrac{1}{n}))\calH,
\quad\text{and}\quad
A_0 | E((\tfrac{1}{n},1))\mathcal H,\qquad n\in\N,
\end{equation*}
which also satisfy conditions (i)--(iii). 
We use \cite[Lemma 2.5]{PST12} (with $k=n=2$) to estimate
\begin{equation}\label{ojadanke}
\begin{split}
\|(A_0|E(\Delta_n)\calH)^2\|\,&\le\,4 \big(M+\| A_0|E(\Delta_n)\calH\| \big)r(A_0|E(\Delta_n)\calH)
\,\leq \frac{4(M+\|A_0\|)}{n},
\end{split}
\end{equation}
where $M>0$ is some constant independent of $n$ and $r(A_0|E(\Delta_n)\calH)$ denotes the spectral radius of $A_0|E(\Delta_n)\calH$. Now, let 
$f\in E((-1,1))\calH$ be arbitrary, and define 
\begin{equation*}
u_n := E(\Delta_n)f,\quad v_n:=E((-1,-\tfrac{1}{n}))f,\quad\text{and}\quad w_n := 
E((\tfrac{1}{n},1))f,\qquad n\in\N.
\end{equation*}
As the intervals $(-1,-\tfrac{1}{n})$ are of negative type and the intervals $(\tfrac{1}{n},1)$ 
are of positive type, it follows from Theorem~\ref{locspecfunc}~(c) and the functional calculus for definitizable operators (see, e.g. \cite[Corollary to Theorem 3.1]{L82})
that $A_0\vert E((-1,-\tfrac{1}{n}))\calH$ and 
$A_0 | E((\tfrac{1}{n},1))\mathcal H$ are both non-negative.
Then the sequences $[A_0v_n,v_n]$ and $[A_0w_n,w_n]$ are both non-negative, and thanks to \eqref{ojadanke} we obtain
$$
\left\|A^2\frac{u_n}{\|u_n\|}\right\| = \frac{1}{\|u_n\|}\|(A_0|E(\Delta_n)\calH)^2
u_n\|\,\le\,\frac{4(M+\| A_0\|)}{n},
$$
which tends to zero as $n\to\infty$. Therefore, condition (iii) yields
$
\liminf_{n\to\infty}\,[Au_n,u_n]\,\ge\,0,
$
and from 
$$
[A_0f,f] = [A_0u_n,u_n] + [A_0v_n,v_n]+ [A_0w_n,w_n]
$$
we conclude $[A_0f,f]\ge 0$.
\vskip 0.2cm
\noindent
{\it Step 3.}
It remains to consider
the self-adjoint operator
$$A_\infty = A|E(\ol\R\setminus (-1,1))\calH$$
and to check that $A_\infty$ is non-negative. 
It is clear that the conditions (i)--(iii) hold also for $A_\infty$,
and, in addition, $A_\infty$ is boundedly invertible. 
It follows from \cite[Lemma~2.4]{j03}  that condition (i) remains valid for $A_\infty^{-1}$. 
%Furthermore, (ii) holds as $A_\infty^{-1}$ is bounded. 
Now, observe that $0\not\in\sigma_p(A_\infty^{-1})$ and hence we can apply 
\cite[Corollary 3 of Proposition II.5.2]{L82} 
(note that the implication in \cite[Corollary 3 of Proposition II.5.2]{L82} also holds if $\alpha=0$ is not a critical point but divides the spectra in positive and negative type)
to conclude that $A_\infty^{-1}$ is non-negative.
From the non-negativity of $A_\infty^{-1}$ we obtain that also
 $A_\infty$ is non-negative. 

Now, observe that with respect  to the decomposition 
$$\calH = E((-1,1))\calH\, [\ds]\, E(\ol\R\setminus (-1,1))\calH$$ 
the operator $A$ can be written in the form
\begin{equation}\label{Wiesbaden}
    A = \mat{A_0}00{A_\infty},
\end{equation}
where $A_0$ is non-negative by Step 2 and $A_\infty$ is non-negative
by Step 3. This implies that $A$ is also non-negative and completes the first part of the proof.

\medskip
Now, we prove the necessity of the conditions (i)--(iii).
Assume for this that $A$ is a non-negative operator in $\calH$. 
Then it is well known that the spectrum of $A$ is real, see, e.g. \cite[Chapter VII, Theorem~1.3]{b}.
Next, let $\la\in\sigma(A)\cap\R^+$ and let $(f_n)$ be a sequence
in $\dom A$ with $\|f_n\| = 1$ such that $(A - \la)f_n\to 0$ as $n\to\infty$. Choose $\veps\in (0,\la)$ and set 
$$
P := E((-\veps,\veps)),
$$
where again $E$ denotes the spectral function of $A$. Let 
$$
g_n := Pf_n \quad \mbox{and}\quad  h_n := (I - P)f_n, \quad n \in \mathbb N.
$$
From $(A - \la)g_n\to 0$ as $n\to\infty$ and $\la\in\rho(A|P\calH)$ it follows that $g_n\to 0$ and thus $\|h_n\|\to 1$. Since $0\in\rho(A|(I - P)\calH)$, the operator $A|(I-P)\calH$ is uniformly positive and thus, for some $\gamma > 0$,
$$
\la [h_n,h_n] = [Ah_n,h_n] - [(A - \la)h_n,h_n]\,\ge\,\gamma\|h_n\|^2 - [(A - \la)h_n,h_n].
$$
 This implies
$$
\liminf_{n\to\infty}\,[f_n,f_n] = \liminf_{n\to\infty}\,[h_n,h_n]\,\ge\,\gamma/\la > 0.
$$
By a similar reasoning we see $\sigma(A)\cap\R^- \subset\sigma_-(A)$, and hence (i)
follows.

Next, we verify the growth of the resolvent of $A$ at $0$ and $\infty$ in (ii) and (iii).
For this we consider the bounded non-negative operator 
$$A_0 = A|E((-1,1))\calH$$ 
and we set $B := JA_0$, where $J$ is a fundamental symmetry in $\mathcal H$. 
Then $B$ is self-adjoint and non-negative in the Hilbert space $(\calH,\hproduct)$ with $\hproduct = [J\cdot,\cdot]$, and, hence,
$B^{1/2}JB^{1/2}$ is self-adjoint in $(\calH,\hproduct)$.
Recall that 
for bounded operators $S$ and $T$ one has $\sigma(ST)\setminus\{0\} = \sigma(TS)\setminus\{0\}$ and
$$
(ST - \la)^{-1} = \la^{-1}\big(S(TS - \la)^{-1}T - I\big),\quad\la\in\rho(ST)\setminus\{0\}.
$$
Therefore, setting 
$S = JB^{1/2}$ and $T = B^{1/2}$ gives $A_0=ST$ and thus
$\rho(A_0)\setminus\{0\}=\rho(B^{1/2}J B^{1/2})\setminus\{0\}$.
For
$\la\in\rho(A_0)\setminus\mathbb R$ we estimate
\begin{equation}\label{hurra}
\begin{split}
\|(A_0 - \la)^{-1}\|
&=\big\| \la^{-1}\big(JB^{1/2}(B^{1/2}JB^{1/2} - \la)^{-1}B^{1/2} - I\big)\big\|\\
&\le \frac{1}{|\la|}\left(\|B^{1/2}\|\|(B^{1/2}JB^{1/2} - \la)^{-1}\|\|B^{1/2}\| + 1\right)\\
&\le \frac{\|B\|}{|\la||\Im\la|} + \frac{1}{|\la|}\le \frac{\|B\|}{|\Im\la|^2} + \frac{1}{|\Im\la|}.
\end{split}
\end{equation}
Thus, the growth of the resolvent of $A_0$ at $0$, and hence that of $A$, is of order $2$. 

For the growth of the resolvent of $A$ at $\infty$ let 
$$A_\infty = A|E(\ol\R\setminus (-1,1))\calH,$$
and observe that $A_\infty$ is non-negative, boundedly invertible, and $A_\infty^{-1}$ is also non-negative. 
For $\la\in\rho(A_\infty)\setminus\R$ we have $A_\infty - \la = -\la (A_\infty^{-1} - \la^{-1})A_\infty$, and hence
$$
(A_\infty - \la)^{-1} = -\la^{-1}A_\infty^{-1}(A_\infty^{-1} - \la^{-1})^{-1} = -\la^{-1} - \la^{-2}(A_\infty^{-1} - \la^{-1})^{-1}.
$$
Using the estimate \eqref{hurra} for $(A_\infty^{-1} - \la^{-1})^{-1}$ we find for some $C>0$
$$
\|(A_\infty - \la)^{-1}\|\,\le\,\frac{1}{|\la|} + \frac{1}{|\la|^{2}} \frac{C}{|\Im\la^{-1}|^2}
=\frac{1}{|\la|} + C \frac{ |\la|^2}{|\Im\la|^2}
\le \frac{1}{|\Im\la|} + C \frac{ |\la|^2}{|\Im\la|^2}.
$$
Choose $\calU:=\{z\in \mathbb C : |z|>1\}$, then $\calU$ is an open neighborhood  of $\infty$ in $\ol\C$ and we conclude for $\lambda \in \calU\setminus\R$
$$
\|(A_\infty - \la)^{-1}\|\,\le\,
  \frac{|\Im\la|+C |\la|^2}{|\Im\la|^2}
  \,\le\,
  \frac{|\la|+C |\la|^2}{|\Im\la|^2}
  \,\le\,
  \frac{(C+1) |\la|^2}{|\Im\la|^2},
$$
which implies that the growth of the resolvent of $A_\infty$, and hence that of $A$, at $\infty$ is of order $2$. Thus, we have shown the resolvent growth in (ii) and (iii).

Finally, note that the second part in  (iii) is clear as $A$ is non-negative.
\end{proof}

In the next theorem we simplify
condition (iii) in Theorem \ref{t:nn} under the assumption that $0$ is a regular critical point by replacing the approximative eigensequences by vectors in the algebraic eigenspace.
Note that a growth of the resolvent of $A$ at $0$ of order $2$ implies $\ker A^3 = \ker A^2$ (and hence
$\ker A^{n+1} = \ker A^n$ for $n\geq 2$). In fact, if $A^3f=0$ and $\lambda\in\rho(A)$, then $-(A-\la)A^2f = \la A^2f$ and hence 
\begin{align*}
-A^2f
&= \la(A-\la)^{-1}A^2f = \la Af + \la^2(A-\la)^{-1}Af = \la Af + \la^2 f + \la^3(A-\la)^{-1}f.
\end{align*}
Thus, for $\la = i\eta$, $\eta>0$, we have $\|A^2f\|\le \eta\|Af\| + \eta^2\|f\| + \eta^3\frac{M}{\eta^2}\|f\|\to 0$ as $\eta\downto 0$, so that $f\in\ker A^2$. As a consequence, the algebraic eigenspace of $A$ at $0$ coincides with $\ker A^2$.

\begin{thm}\label{t:nn2}
A self-adjoint operator $A$ in the Krein space $(\calH,\product)$ is non-negative, and $0$ is not a singular critical point of $A$ if and only if the following conditions are satisfied:
\begin{enumerate}
\item[{\rm (i)}]   $\sigma(A)\subset\R$ and $\sigma(A)\cap\R^\pm\subset\sigma_\pm(A)$.
\item[{\rm (ii)}]  The growth of the resolvent of $A$ at $\infty$ is of order $2$.
\item[{\rm (iii)}] The growth of the resolvent of $A$ at $0$ is of order $2$ and
\footnote{Note that conditions (i), (ii), and the assumption that the growth of the resolvent of $A$ at $0$ is of order $2$ in (iii) already imply that 
$A$ is definitizable (cf.\ Step 1 of the proof of Theorem \ref{t:nn}).
Thus, the notion of critical points is meaningful, and we may exclude $0$ as a singular critical point in (iii).}
$0$ is not a singular critical point of $A$, and $[Af,f]\ge 0$ for $f\in \ker A^2$.
\end{enumerate}
\end{thm}
\begin{proof}
Assume that (i)--(iii) are satisfied. 
As in Step 1 of the proof of Theorem \ref{t:nn}, it follows that the operator 
$A$ is definitizable with the only possible critical points $0$ and $\infty$. By assumption, $0$ is not a singular critical point. 
The same arguments as in Step 3 of the proof of Theorem~\ref{t:nn} show that 
$$A_\infty = A|E(\ol\R\setminus (-1,1))\calH$$ is non-negative.
Therefore, it remains to prove that the bounded operator $$A_0 = A|E((-1,1))\calH$$ 
is non-negative. Since $0$ is not a singular critical point of $A$, we may apply \cite[Theorem II.5.7]{L82}, which shows that the spectral function $E$ of $A$ admits an extension to Borel subsets of $\R$ with $0$ in their $\R$-boundary. Therefore, with respect to the decomposition 
$$
    E((-1,1))\calH = E(\{0\})\calH\,[\ds]\, E((-1,1)\setminus\{0\})\calH
$$
the operator $A_0$ can be written as a diagonal operator matrix
\begin{equation}\label{Martinroda2}
A_0 = \mat{A_0^\prime}00{A_0^{\prime\prime}},
\end{equation}
where $A_0^\prime = A|E(\{0\})\calH$ and $A_0^{\prime\prime} = A|E((-1,1)\setminus\{0\})\calH$. Note that $E(\{0\})\calH = \ker A^2$ (cf.\  \cite[Theorem II.5.7]{L82}). Therefore, $A_0^\prime$ is non-negative by (iii). Moreover, $A_0''$ is non-negative by \cite[Corollary 3 of Proposition II.5.2]{L82}.
Hence, $A_0$ is a non-negative operator, and thus $A$ is non-negative.

Conversely, if $A$ is non-negative and $0$ is not a singular critical point of $A$, 
then (i)--(iii) follow from Theorem~\ref{t:nn}.
\end{proof}

The next theorem discusses the important special case that the non-negative operator $A$ in $(\calH,\product)$
is similar to a self-adjoint operator in a Hilbert space.
Recall that a linear operator $T$ in the Krein space $(\calH,\product)$ is said to be \emph{similar to a self-adjoint operator in a Hilbert space} if there exists a bounded and boundedly invertible operator $V$ in $\calH$ such that $VTV^{-1}$ is self-adjoint in $(\calH,\hproduct)$. This is in fact equivalent to $T$ itself being self-adjoint in a Hilbert space $(\calH,\aproduct)$, where $\aproduct$ induces the Krein space topology. 
Indeed, it is straightforward to verify that for a bounded and boundedly invertible operator $V$ in $\calH$, $VTV^{-1}$ is self-adjoint in $(\calH,\hproduct)$ if and only if $T$ is self-adjoint in $(\calH,(V\cdot,V\cdot))$.

\begin{thm}\label{t:nn3}
A self-adjoint operator $A$ in the Krein space $(\calH,\product)$ is non-negative with $0$ and $\infty$ 
not being singular critical points of $A$ and $\ker A = \ker A^2$ if and only if the following conditions are satisfied:
\begin{enumerate}
\item[{\rm (i)}]   $\sigma(A)\subset\R$ and $\sigma(A)\cap\R^\pm\subset\sigma_\pm(A)$.
\item[{\rm (ii)}]  $A$ is similar to a self-adjoint operator in a Hilbert space.
\end{enumerate}
\end{thm}
\begin{proof}
Suppose that $A$ is non-negative with $0$ and $\infty$ not being singular critical points and $\ker A = \ker A^2$. 
Then, by Theorem~\ref{t:nn}, (i) follows and it remains to show (ii).
For this, we set 
$$
\calH_0 = E((-1,1))\calH \quad \mbox{and} \quad  \calH_\infty = E(\ol\R\setminus (-1,1))\calH
$$
as well as 
$$
A_0 = A\,|\,\calH_0\quad  \mbox{and} \quad A_\infty = A\,|\,\calH_\infty.
$$
Consider $A_0$ and decompose $\calH_0$ further into 
$$\calH_0' = E(\{0\})\calH = \ker A \quad \mbox{and} \quad \calH_0'' = E((-1,1)\setminus\{0\})\calH.
$$
By \cite[Theorem 5.7]{L82} and the subsequent discussion, the projectors $E_- = E((-1,0))$ and $E_+ = E((0,1))$ exist, and
\[
\big(\calH_0'',\,[(E_+ - E_-)\cdot,\cdot]\big)
\]
is a Hilbert space in which $A_0'' = A_0\,|\,\calH_0''$ is self-adjoint, since $A_0''$ commutes with $E_+$ and $E_-$. Hence, if $\ker A = \{0\}$, then $A_0 = A_0''$ is self-adjoint in a Hilbert space. If $\ker A\neq\{0\}$, we choose a fundamental symmetry $J_0'$ on $\calH_0'$. Then
\[
J_0 = \begin{pmatrix}J_0' & 0\\0 &  E_+-E_-\end{pmatrix}
\]
defines a fundamental symmetry on $\calH_0 = \calH_0'\,[\ds]\,\calH_0''$, and
$A_0\,|\,\calH_0'$ is the zero operator. Therefore  $A_0$ is self-adjoint in the Hilbert space $(\calH_0,[J_0\cdot,\cdot])$. 

Note that $A_\infty$ is boundedly invertible and that $A_\infty^{-1}$ is non-negative, $0$ is not a singular critical point of $A_\infty^{-1}$, and $\ker A_\infty^{-1} = \{0\}$. Hence, as just proved for $A_0$, there exists a fundamental symmetry $J_\infty$ such that $A_\infty^{-1}$ is self-adjoint in $(\calH_\infty,[J_\infty\cdot,\cdot])$. The same holds for $A_\infty$, and thus $A$ is self-adjoint in the Hilbert space $(\calH,[J_A\cdot,\cdot])$, where $J_A = \smallmat{J_0}{0}{0}{J_\infty}$.

Conversely, assume that (i) and (ii) hold, let $\<\cdot\,,\cdot\>$ be the Hilbert space scalar product on $\calH$ such that $A$ is self-adjoint, and denote by $\|\cdot\|_1$ the corresponding norm. As a self-adjoint operator in a Hilbert space, the resolvent growth of $A$ is of order $1$ at each point $\la_0\in\ol\R$, and $\ker A = \ker A^2$. Hence, the conditions (i) and (ii) of Theorem~\ref{t:nn2} are satisfied. To prove Theorem \ref{t:nn2}~(iii), observe first that $[Af,f]=0$ for all $f\in\ker A^2 = \ker A$, and hence for Theorem \ref{t:nn2}~(iii) it remains to verify that $0$ is not a singular critical point of $A$. For this, denote by $\wh E$ the spectral measure of $A$ as a self-adjoint operator in $(\calH,\aproduct)$ and by $E$ the spectral function of $A$ as a definitizable operator. By the uniqueness of the spectral function (see, e.g., \cite[Lemma 2.12]{j03}), we have $\wh E(\Delta) = E(\Delta)$ for all $\Delta$ for which $E(\Delta)$ is defined. Moreover, as $\|\cdot\|_1$ and $\|\cdot\|$ are equivalent norms and $\|\wh E(\Delta)f\|_1\le \|f\|_1$ for each $f\in\calH$ and each Borel set $\Delta\subset\ol\R$, we conclude that there exists $C>0$ such that $\|E(\Delta)\|\le C$ for each $\Delta$ for which $E(\Delta)$ is defined. Hence, $0$ is not a singular critical point of $A$. The same argument also shows that $\infty$ is not a singular critical point of $A$. Therefore, conditions (i)--(iii) in Theorem~\ref{t:nn2} are satisfied, and the assertion follows.
\end{proof}

\begin{rem}
In many applications in $\mathcal P \mathcal T$-quantum mechanics \cite{HollaDieWaldFee1,LT04,Mosta10}, one is interested in the existence of a so-called $\mathcal C$-metric operator which is equivalent to the similarity of the underlying Hamiltonian operator $A$ to a self-adjoint operator in a Hilbert space. Here, we only refer to \cite{HollaDieWaldFee2,HollaDieWaldFee3}.
\end{rem}

\section{Spectral characterization of locally non-negative operators}
Inspired by the local nature of the characterizations of non-negative operators in the previous section, we study a class of self-adjoint operators in Krein spaces whose spectral properties resemble those of non-negative operators outside a compact set.
In this sense, the next definition  can be viewed as a localization of the notion of 
non-negati\-vi\-ty of self-adjoint operators in a neighborhood of $\infty$.

\begin{defn}\label{d:lnn}
Let $K\subset\C$ be a compact set which is symmetric with respect to the real axis
such that 
$\C^+\setminus K$ is simply connected. A self-adjoint operator $A$ in the Krein space $(\calH,[\cdot,\cdot])$ is said to be {\em non-negative over} $\ol\C\setminus K$ if for any bounded open neighborhood $\calU$ of $K$ in $\C$ with $0\notin\partial\calU$ there exists a bounded $\product$-self-adjoint projection $E_\infty$ such that with respect to the decomposition
\begin{equation}\label{e:dec_lnn}
\calH = (I-E_\infty)\calH\,[\ds]\,E_\infty\calH
\end{equation}
the operator $A$ can be written as a diagonal operator matrix
\begin{equation}\label{e:deco}
A = \mat{A_b}00{A_\infty},
\end{equation}
where $A_b$ is a bounded self-adjoint operator in the Krein space $((I-E_\infty)\calH ,\product)$ with
$\sigma(A_b) \subset \ol\calU$ and $A_\infty$ is a non-negative operator in the Krein space $(E_\infty\calH ,\product)$ with $\calU\subset\rho(A_\infty)$.
\end{defn}

\begin{rem}
Definition \ref{d:lnn} is slightly more general than the definition in \cite{BPT28} as we do not assume here that $0\in K$.
Furthermore, it differs slightly from
\cite[Definition~3.1]{BJ05}; for more details, see \cite[Definition 2.1 and footnote]{BPT28}.
\end{rem}

Observe that a self-adjoint operator $A$ in $\calH$ is non-negative over $\ol\C = \ol\C\setminus\emptyset$ if and only if $A$ is non-negative (in the sense of \eqref{ojabitte3}). Indeed, if $A$ is non-negative, then for any bounded open set $\calU\subset\C$ with $0\notin\partial\calU$ define $E_\infty := I - E(\calU\cap\R)$, where $E$ denotes the spectral function of $A$. Then $A$ decomposes as in \eqref{e:dec_lnn}--\eqref{e:deco} with the desired properties. Conversely, if $A$ is non-negative over $\ol\C$, simply choose $\calU = \emptyset$ in Definition \ref{d:lnn} to see that $A$ is non-negative.

\begin{prop}
If the self-adjoint operator $A$ in the Krein space $(\calH,\product)$ is non-negative over $\ol\C\setminus K$, then there exists $\gamma\in\R$ such that
\begin{equation*}
[A f,f] \ge\gamma\|f\|^2\quad  \mbox{for all } f\in\dom A.
\end{equation*}
\end{prop}
\begin{proof}
Let $E_\infty$ be a $\product$-self-adjoint projection such that $A$ can be written as in \eqref{e:deco} with respect to the decomposition \eqref{e:dec_lnn} with a bounded operator $A_b$ and a non-negative operator $A_\infty$. 
Since $E_\infty\calH$ and $(I-E_\infty)\calH$ are both Krein spaces they admit  fundamental decompositions
$$
E_\infty\calH =\calH_\infty^+\,[\ds]\,\calH_\infty^-
\quad \mbox{and} \quad 
(I-E_\infty)\calH =\calH_b^+\,[\ds]\,\calH_b^-,
$$
that lead to
a new fundamental decomposition of the underlying space 
$\calH$,
$$
\calH =\left(\calH_\infty^+\,[\ds]\,\calH_b^+\right)
\,[\ds]\,
\left(\calH_\infty^-\, [\ds]\,\calH_b^-\right).
$$
The corresponding fundamental symmetry $J$ satisfies $JE_\infty = E_\infty J$, hence $(I-E_\infty)\calH$ and $E_\infty\calH$ are orthogonal with respect to the scalar product $[J\cdot,\cdot]$.
Then for $f\in\dom A$, $f = f_b + f_\infty$ with $f_b\in (I-E_\infty)\calH$ and $f_\infty\in E_\infty\calH$, we have $$\|f\|^2=\|f_b\|^2+\|f_\infty\|^2,$$
and hence we obtain
$$
[Af,f] = [A_bf_b,f_b] + [A_\infty f_\infty,f_\infty]\,\ge\,-|[A_bf_b,f_b]|\,\ge\,-\|A_b\|\|f_b\|^2\,\ge\,-\|A_b\|\|f\|^2.
$$
Thus, the claim holds with $\gamma = -\|A_b\|$.
\end{proof}

The next statement  on local definitizability of locally non-negative operators
is a consequence of the representation \eqref{e:deco} and Theorem \ref{t:nn}.

\begin{prop}\label{Halle}
Let $A$ be a self-adjoint operator in the Krein space $(\calH,\product)$ which is non-negative over $\ol\C\setminus K$.
Then $A$ is definitizable  over $\ol\C\setminus K$, and we have
$$
\sigma(A)\setminus\R\subset K \quad \mbox{and}\quad (\sigma(A)\setminus K)\cap\R^\pm\subset\sigma_\pm(A).
$$
\end{prop}

%It might now be natural to conjecture that conversely, if $[Af,f]\ge\gamma\|f\|^2$ for all $f\in\dom A$ and $\rho(A)\neq\emptyset$, then $A$ is non-negative over some $\ol\C\setminus K$. We leave this as an open question.

% \begin{ex}
% Let $\calH = \C^2$, $J = \smallmat{1}{0}{0}{-1}$, and $A_s := \smallmat{s}{s+1}{-s-1}{-s}$, $s > 0$. Then, since $JA_s = \smallmat{s}{s+1}{s+1}{s}$, the operator $A_s$ is self-adjoint in the Krein space $(\C^2,\product)$ for every $s > 0$, where $\product := (J\cdot,\cdot)$. Moreover, for $f = (x,y)^T$, $x,y\in\C$, we have
% \begin{align*}
% [A_sf,f]
% &= (JA_sf,f) = \left(\vek{sx+(s+1)y}{(s+1)x + sy},\vek{x}{y}\right) = s|x|^2 + (s+1)2\Re(y\ol x) + s|y|^2\\
% &= 2\Re(y\ol x) + s|x+y|^2\,\ge\,2\Re(y\ol x)\,\ge\,-2|xy|\,\ge\,-\|f\|^2.
% \end{align*}
% But the eigenvalues of $A_s$ are $\pm i\sqrt{2s+1}$.
% \end{ex}

If $A$ is non-negative over $\ol\C\setminus K$, then Proposition \ref{Halle} and  Theorem \ref{locspecfunc} ensure that
$A$ possesses a local spectral function $E$ on $\ol\R\setminus K$. 
The spectral projection $E(\Delta)$ is defined for all Borel sets $\Delta\subset\ol\R\setminus K$ for which neither $\infty$ nor the boundary points of $K\cap\R$ (nor $0$, if $0\notin K$) are boundary points.

In the next theorem, which is a local version of Theorem~\ref{t:nn}, we provide a characterization of self-adjoint operators that are locally non-negative.

\begin{thm}\label{t:eqthm}
Let $A$ be a self-adjoint operator in the Krein space $(\calH,\product)$
and let 
$K\subset\C$ be a compact set which is symmetric with respect to the real axis
such that $\C^+\setminus K$ is simply connected. Then $A$ is non-negative over $\ol\C\setminus K$ if and only if the following conditions are satisfied:
\begin{enumerate}
\item[{\rm (i)}]   $\sigma(A)\setminus\R\subset K$ and $(\sigma(A)\setminus K)\cap\R^\pm\subset\sigma_\pm(A)$.
\item[{\rm (ii)}]  The growth of the resolvent of $A$ at $\infty$ is of order $2$.
\item[{\rm (iii)}] If $0\notin K$, then the growth of the resolvent of $A$ at $0$ is of order $2$ and for each sequence $(f_n)$ in $\dom A^2$ with $A^2f_n/\|f_n\|\to 0$ as $n\to\infty$ we have
$$
\liminf_{n\to\infty}\,[Af_n,f_n]\ge 0.
$$
\end{enumerate}
\end{thm}
\begin{proof}
Assume that conditions (i)--(iii) are satisfied. Then it follows in the same way as in Step 1 of the proof 
of Theorem~\ref{t:nn} that $A$ is definitizable over $\ol\C\setminus K$.
In fact, from (i) it is clear that for every point 
$\mu\in \R^+\setminus K$ ($\mu\in\R^-\setminus K$) there exists an open neighborhood
in $\R^+$ ($\R^-$) which is of positive type (negative type, respectively), 
and for $\infty$ there exists a neighborhood where both components are of definite (but different) type. The same is true for $0$ if $0\not\in K$. Furthermore, the growth of the resolvent of $A$ near points of $\sigma(A)\setminus K$ is of order one and by
(ii) of order $2$ near $\infty$; the same holds for $0$ if $0\not\in K$ by (iii). Thus, $A$ is definitizable over $\ol\C\setminus K$, 
and hence possesses a local spectral function $E$ on $\ol\R\setminus K$. 
Now, consider some bounded open neighborhood $\calU$ of $K$ in $\C$ with $0\notin\partial\calU$. If we define $E_\infty := E(\ol\R\setminus\calU)$, then 
with respect to the space decomposition
$$
\calH = (I - E_\infty)\calH\,[\ds]\,E_\infty\calH,
$$
the operator $A$ admits the diagonal form
$$
A = \mat{A_b}00{A_\infty}.
$$
By Theorem \ref{locspecfunc} (b) the spectrum of the operator $A_\infty$ is real and $\calU\subset\rho(A_\infty)$. It follows from (i)--(iii) and Theorem~\ref{t:nn} that the operator $A_\infty$ is non-negative in $E_\infty\calH$.
By Theorem \ref{locspecfunc} (e), the operator $A_b= A|(I-E_\infty) \calH$ is bounded and,
by Theorem \ref{locspecfunc} (d), we have 
$$\sigma(A_b) \subset \bigl(\sigma(A) \setminus
(\ol\R \setminus \ol\calU)\bigr)\subset \ol\calU.$$
Therefore, we have shown that $A$ is non-negative over $\ol\C\setminus K$.

Conversely, if $A$ is non-negative over $\ol\C\setminus K$, then (i)--(iii) are consequences of the representation \eqref{e:deco}
and Theorem \ref{t:nn}. 
\end{proof}

In the same way Theorem \ref{t:nn2} allows for a version for locally non-negative operators using the decomposition~\eqref{e:deco}. We omit the details.

\section{Perturbations of non-negative operators}\label{555}
In this section, we shall see that locally non-negative operators appear naturally as perturbations of non-negative operators. 
Let $(\calH,\product)$ be a Krein space, fix a fundamental symmetry $J$ and consider the corresponding fundamental decomposition
\begin{equation} \label{Baku}
    \calH = \calH_+ \,[\ds]\, \calH_-,
\end{equation}
which is orthogonal with respect to $\product$, and  $(\calH_\pm , \pm\product)$
are both Hilbert spaces. 
The norm induced by the Hilbert space scalar product $[J\cdot,\cdot]$ will be denoted by $\|\cdot\|$.

Let us start with a particularly simple situation: Bounded perturbations of non-negative operators which are diagonal with respect to the decomposition \eqref{Baku}.
More precisely, let $A$ be a non-negative operator in $\calH$ of the form
\begin{equation}\label{Bakuu}
A = \mat{A_+}00{A_-};
\end{equation}
then $A_+$ is a non-negative operator in the Hilbert space $(\calH_+, \product)$ and $A_-$ is non-positive in the Hilbert space $(\calH_-, -\product)$. Let $V$ be a bounded self-adjoint operator in the Krein space $(\calH,\product)$. Then $V$ has the form 
\begin{equation}\label{Bakuuu}
V= \mat{V_+}{V_0}{-V_0^*}{V_-}
\end{equation}
with respect to the decomposition \eqref{Baku},
where $V_\pm$ are bounded self-adjoint operators in $\calH_\pm$ and 
$V_0$ is a bounded operator mapping from $\calH_-$ to $\calH_+$. It turns out in the next result that the perturbed operator $A+V$ is locally non-negative; here, the compact set $K$ is specified explicitly in terms of the operator norms of $V_\pm$ and $V_0$. For later purposes we emphasize that the operator norm is defined via the norm $\|\cdot\|$ induced by $[J\cdot,\cdot]$.

\begin{thm}\label{l:lmm}
The operator $A+V$ is self-adjoint in the Krein space
$(\calH,\product)$ and  non-negative over $\ol\C\setminus K$, where 
$$
K = \left\{z\in\C : \dist(z,[-\|V_+\|,\|V_-\|])\,\le\,\|V_0\|\right\}.
$$
\end{thm}
\begin{proof}
The operator $A_++V_+$ is semibounded from below by $-\|V_+\|$ in the Hilbert space $(\calH_+, \product)$ and the operator $A_-+V_-$ is semibounded from above by $\|V_-\|$ in the Hilbert space $(\calH_-, -\product)$. Hence,
$$
\sigma(A_++V_+) \cap \sigma(A_-+V_-)\subset [-\|V_+\|,\|V_-\|]\subset K,
$$
and the statement follows from \cite[Theorem 3.5]{BPT28} and Theorem \ref{t:eqthm}.
\end{proof}

\begin{rem}
Results in the flavor of
Theorem~\ref{l:lmm} have a long history and go back (at least) to the paper \cite{L67}, where a variant 
for bounded operators in Krein spaces under a compactness assumption was proved, 
see also \cite{lmm} and \cite{LLMT05}, as well as  
\cite[Proposition 2.6.8]{T08} and \cite [Theorem 5.5]{T09} for more general 
versions without compactness assumptions. The present version of 
Theorem~\ref{l:lmm} can be viewed as a variant of \cite[Theorem 3.5]{BPT28}, which is slightly stronger 
providing spectral estimates in terms of $\sigma(A_+)$ and $\sigma(A_-)$ explicitly. Furthermore, we mention the more recent generalizations in \cite[Theorem 4.3]{GLMPT20} and \cite[Theorem 4.1]{p23}.
\end{rem}

In applications, the situation is often slightly more complicated than above, namely, the unperturbed operator $A$ is non-negative in $\calH$, but not of the simple diagonal form \eqref{Bakuu}. Then, in general, bounded perturbations $V$ may lead to perturbed operators $A+V$, where the spectrum covers the full complex plane (see, e.g. \cite[Example 3.2]{BPT28}) and thus it is necessary to impose additional structural conditions on the unperturbed operator $A$ or the perturbation $V$. A natural (and still useful) restriction is to assume that $0$ and $\infty$ are not singular critical points of the non-negative operator $A$, in which case the spectral projections $E(\mathbb R^\pm)$ to $\mathbb R^\pm$ exist. If, in addition, $0$ is not an eigenvalue of $A$, then
\begin{equation}\label{tilde}
    \calH = E(\mathbb R^+)\calH \, [\ds]\, E(\mathbb R^-)\calH
\end{equation}
is also a fundamental decomposition of the underlying Krein space, which (in general) differs from the one in \eqref{Baku}. The corresponding fundamental symmetry $\widetilde J$ can be expressed as
\begin{equation}\label{tau0j}
\widetilde J=\frac 1 \pi\,\lim_{n\to\infty}\,\int_{1/n}^n\,\left((A + it)^{-1} + (A - it)^{-1}\right)\,dt,
\end{equation}
where the limit exists in the strong sense. It is clear that $[\widetilde J\cdot,\cdot]$ is a Hilbert space scalar product and the corresponding norm, which will denoted by $\|\cdot\|_\sim$, is 
equivalent to the norm $\|\cdot\|$ used above (see, e.g., \cite[Proposition 1.2]{L82}). As a result, the unperturbed non-negative operator $A$ admits a  representation as in \eqref{Bakuu} with respect to the fundamental decomposition \eqref{tilde}.
Therefore, if the perturbation $V$ is a bounded self-adjoint operator of the form 
\eqref{Bakuu} with
respect to the fundamental decomposition \eqref{tilde}, then Theorem~\ref{l:lmm} implies that $A+V$ is self-adjoint in the Krein space
$(\calH,\product)$ and  non-negative over $\ol\C\setminus K$, where 
\begin{equation}\label{khier}
 K = \left\{z\in\C : \dist(z,[-\|V_+\|_\sim,\|V_-\|_\sim])\,\le\,\|V_0\|_\sim\right\}.
\end{equation}
However, the aim is to express $K$ in terms of the operator norm induced via 
$\|\cdot\|$ and the fundamental symmetry $J$ corresponding to the original fundamental decomposition \eqref{Baku}. For this, the operator norm of $\widetilde J$ in \eqref{tau0j} (in terms of the original norm $\|\cdot\|$ induced by $[J\cdot,\cdot]$),
\begin{equation}\label{tau0}
\tau = \| \widetilde J\| =
\frac 1 \pi\left\|\lim_{n\to\infty}\,\int_{1/n}^n\,\left((A + it)^{-1} + (A - it)^{-1}\right)\,dt\,\right\|,
\end{equation}
is needed and, roughly speaking, leads to a different form of \eqref{khier}. 
This is made more precise in the next result,
which was proved in \cite[Theorem 3.1]{BPT28}.
We note that the quantity $\tau$ in \eqref{tau0} is, in general, smaller than the 
corresponding quantity in \cite[(3.2) and (3.17)]{BPT28} (see also \cite[(5.2)]{p23}), and hence leads to better estimates.

\begin{thm}\label{t:main1}
Let $A$ be a non-negative operator in $(\calH,\product)$ such that $0$ and $\infty$ are not singular critical points of $A$, respectively, assume that $0\notin\sigma_p(A)$, and let
$\tau$ be as in \eqref{tau0}. Furthermore, let $V$
be a bounded self-adjoint operator in $(\calH,\product)$ and let $J$ be the fundamental symmetry in \eqref{Baku}. Then
 $A+V$ is self-adjoint, and the following statements hold:
\begin{itemize}
\item [{\rm (i)}] If $V$ is non-negative, then $A+V$ is non-negative.
\item [{\rm (ii)}] If $V$ is not non-negative, then $A+V$ is
 non-negative over $\ol\C\setminus K$, where 
$$
 K = \left\{z\in\C : \dist(z,[-d,d])\,\le\,\tfrac{1+\tau}{2}\|V\| \right\}\quad\text{and}\quad 
 d = -\tfrac{1+\tau}{2}\,\min\sigma(JV).
$$
\end{itemize}
\end{thm}

For unbounded operators $A$ it is desirable to allow unbounded perturbations $V$, which of course
need to satisfy suitable additional conditions in order to conclude self-adjointness and further 
spectral properties of $A+V$. In the spirit of the classical Kato-Rellich theorem a natural 
assumption is relative boundedness of $V$ with respect to the unperturbed operator $A$.
This brings us to a more recent result from \cite{p23}. We denote the closed ball in $\C$ with center $z\in\C$ and radius $r > 0$ by $B_r(z)$.

\begin{thm}\label{t:caot_main}
Let $A$ be a non-negative operator in $(\calH,\product)$ such that $0$ and $\infty$ are not singular critical points of $A$, assume that $0\notin\sigma_p(A)$, and let $\tau$ be as in \eqref{tau0}. Furthermore, let  $V$ be a symmetric operator in $(\calH,\product)$ with $\dom A\subset\dom V$ such that
\[
(1+\tau)\tau\|Vf\|^2\,\le\,2a\|f\|^2 + b\|A f\|^2,\quad f\in\dom A,
\]
where $a,b\ge 0$, $b < 1$. Then the operator $A+V$ is self-adjoint, and with
$$
\nu := \inf\{[Vf,f] : f\in\dom V,\,\|f\|=1\}\geq -\infty
$$
the following statements hold:
\begin{itemize}
\item [{\rm (i)}] If $\nu\geq 0$, then $A+V$ is non-negative.
\item [{\rm (ii)}] If $\nu\in (-\infty,0)$, then
$A+V$ is non-negative over $\ol\C\setminus K$, where 
$$
 K = \bigcup_{t\in[-\gamma,\gamma]}B_{\sqrt{a + bt^2}}(t)\qquad\text{and}\qquad \gamma = \min\left\{\sqrt{\tfrac{1+\tau}{2\tau}a},\tfrac{1+\tau}{2}|\nu|\right\}.
$$
If, in addition, $b<\frac{\tau-1}{2\tau}$, then 
$A+V$ is non-negative over $\ol\C\setminus K$, where 
$$
K = \bigcup_{t\in[-\gamma,\gamma]}B_{\sqrt{\frac{1+\tau}{2\tau(1-b)}(a+bt^2)}}(t)\qquad\text{and}\qquad \gamma = \min\left\{\sqrt{\tfrac{1+\tau}{2\tau}a},\tfrac{1+\tau}{2}|\nu|\right\}.
$$
\item [({\rm iii)}] If $\nu=-\infty$, then $A+V$ is non-negative over $\ol\C\setminus K$, where 
$$
 K = \bigcup_{t\in[-\gamma,\gamma]}B_{\sqrt{a + bt^2}}(t)\qquad\text{and}\qquad \gamma = 
 \sqrt{\tfrac{1+\tau}{2\tau}a}.
$$
If, in addition, $b<\frac{\tau-1}{2\tau}$, then 
$A+V$ is non-negative over $\ol\C\setminus K$, where 
$$
K = \bigcup_{t\in[-\gamma,\gamma]}B_{\sqrt{\frac{1+\tau}{2\tau(1-b)}(a+bt^2)}}(t)\qquad\text{and}\qquad \gamma = 
 \sqrt{\tfrac{1+\tau}{2\tau}a}.
$$
\end{itemize}
Furthermore, in all cases {\rm (i)--(iii)}, $\infty$ is not a singular critical point of $A+V$.
\end{thm}

We remark that in the case of a bounded perturbation $V$, we have $b=0$ and $a = \frac{(1+\tau)\tau}2\|V\|^2$, and the statement (ii) in Theorem \ref{t:caot_main} yields the exact same result as Theorem \ref{t:main1}; cf.\ \cite[Remark 5.3]{p23}.

%\subsection*{Data Availability Statement}
 %No data were collected, generated or consulted
%in connection with this research.
%
%\subsection*{Declarations}
%
%\subsection*{Conflicts of Interest}
 %The authors have no Conflict of interest to declare that
%are relevant to the content of this article.

\section{Author Affiliations}
\end{document}